\documentclass[a4paper,12pt,draft]{amsart}
\usepackage{verbatim}
\newtheorem{theorem}{Theorem}[section]
\newtheorem{lemma}[theorem]{Lemma}
\newtheorem{proposition}[theorem]{Proposition}
\newtheorem{corollary}[theorem]{Corollary}

\theoremstyle{definition}
\newtheorem{definition}{Definition}[section]

\newtheorem*{fact}{Fact}

\theoremstyle{remark}

\numberwithin{equation}{section}

\def\N{{\mathbb
N}}

\DeclareMathOperator{\spn}{span}
\DeclareMathOperator{\spann}{span}

\newcommand{\conv}{\textnormal{conv}}

\newcommand{\calL}{{\mathcal
L}}

\newcommand{\diam}{\hbox{diam}\,}

\newcommand{\xast}{x^{\ast}}

\newcommand{\yast}{y^{\ast}}

\newcommand{\Xast}{X^{\ast}}

\newcommand{\Ts}{T^{\ast}}
\newcommand{\Sast}{S^{\ast}}

\newcommand{\us}{u^{\ast}}
\newcommand{\vs}{v^{\ast}}

\newcommand{\Xs}{X^{\ast}}
\newcommand{\xs}{x^{\ast}}
\newcommand{\Xss}{X^{\ast\ast}}
\newcommand{\xss}{x^{\ast\ast}}
\newcommand{\Ys}{Y^{\ast}}
\newcommand{\ys}{y^{\ast}}

\newcommand{\eps}{\varepsilon}

\newcommand{\lxy}{\mathcal{L}(X,Y)}

\begin{document}

\title{Rough norms in spaces of operators}
\date{26.04.2013}

\author{Rainis Haller}
\curraddr{}
\thanks{This research was supported by institutional research funding IUT20-57 of the Estonian Ministry of Education and Research.}

\author{Johann Langemets}
\curraddr{}
\thanks{}

\author{M\" art P\~oldvere}
\curraddr{}
\thanks{}



\date{}

\dedicatory{}

\begin{abstract}
We investigate sufficient and necessary conditions for the space of bounded linear operators between two Banach spaces to be rough or average rough. Our main result is that $\calL(X,Y)$ is $\delta$-average rough whenever $X^\ast$ is $\delta$-average rough and $Y$ is alternatively octahedral. This allows us to give a unified improvement of two theorems by Becerra Guerrero, L\'{o}pez-P\'{e}rez, and Rueda Zoca [J. Math. Anal. Appl. 427 (2015)].

\end{abstract}

\maketitle
\section{Introduction}\label{sec: 0. Introduction}
All Banach spaces considered in this paper are non-trivial and over the real field. The closed unit ball of a Banach space $X$ is denoted by $B_X$ and its unit sphere by $S_X$. The dual space of $X$ is denoted by $X^\ast$, and the Banach space of all bounded linear operators acting from $X$ to another Banach space $Y$ by $\calL(X,Y)$.

\begin{definition}
Let $X$ be a Banach space and $\delta>0$. The space $X$ is said to be
\begin{itemize}
\item \emph{$\delta$-rough} \cite{leach_differentiable_1972} if, for every $x \in S_X$,
\[
\limsup_{\|y\|\to0} \frac{\|x+y\|+\|x-y\|-2}{\|y\|}\geq\delta;
\]
\item \emph{$\delta$-average rough} \cite{deville_dual_1988} if, whenever $n\in\N$ and $x_1,\dotsc,x_n \in S_X$,
\[
\limsup_{\|y\|\to0}\frac{1}{n}\sum_{i=1}^n \frac{\|x_i+y\|+\|x_i-y\|-2}{\|y\|}\geq\delta.
\]
\end{itemize}
The space $X$ is said to be \emph{non-rough}, if there is no $\eps>0$ such that $X$ is $\eps$-rough.
\end{definition}

A dual characterization of roughness is well known. The space $X$ is $\delta$-rough if and only if the diameter of every weak$^{\ast}$ slice of $B_{\Xast}$ is greater than or equal to $\delta$ \cite{john_rough_1978}. The space $X$ is $\delta$-average rough if and only if the diameter of every convex combination of weak$^{\ast}$ slices of $B_{\Xast}$ is greater than or equal to $\delta$ \cite{deville_dual_1988}.
	
 Banach spaces which are $2$-average rough are exactly the \emph{octahedral} ones (see \cite{becerra_guerrero_octahedral_2014}, \cite{deville_dual_1988}, and \cite{godefroy_metric_1989}). A weaker version of octahedrality was introduced in \cite{haller_duality_2015} and it was shown that a Banach space $X$ is \emph{weakly octahedral} if and only if the diameter of every non-empty relatively weak$^{\ast}$ open subset of $B_{\Xast}$ is 2 \cite[Theorem~2.8]{haller_duality_2015}.

\begin{definition}
A Banach space $X$ is said to be
\begin{itemize}
\item \emph{octahedral} (see \cite{godefroy_metric_1989} and \cite[Proposition~2.2]{haller_duality_2015})
		if, whenever $n\in\N$, $x_1,\dots,x_n\in S_X$, and $\eps>\nobreak0$,
		there is~a $y\in S_X$ such that
		\[
		\|x_i+y\|\geq 2-\eps\quad\text{for all $i\in\{1,\dots,n\}$;}
		\]		
\item \emph{weakly octahedral} (see \cite[Proposition~2.6]{haller_duality_2015}) if, whenever $n\in\N$, $x_1,\dots,x_n\in S_X$, $\xs\in B_{\Xs}$, and $\eps>\nobreak0$,	there is~a $y\in S_X$ such that
		\[
		\|x_i+ty\|\geq(1-\eps)\bigl(|\xs(x_i)|+t\bigr)\quad\text{for all $i\in\{1,\dots,n\}$ and $t>0$.}
		\]
		\end{itemize}
\end{definition}

Our note is motivated by the recent paper \cite{becerra_guerrero_octahedral_2015}, where octahedrality of the space of bounded linear operators is studied. In Section~\ref{sec: sufficient conditions}, we give a unified improvement of the following Theorems~\ref{thm: L(X,Y) OH when Y OH} and \ref{thm: L(X,Y) OH when Y has P1} obtained in \cite{becerra_guerrero_octahedral_2015}. Moreover, we study their quantified versions in terms of roughness and average roughness (see Theorem~\ref{thm: L(X,Y) OH when Y has P2}).

\begin{theorem}[{\cite[Theorem~3.5]{becerra_guerrero_octahedral_2015}}]\label{thm: L(X,Y) OH when Y OH}
Let $X$ and $Y$ be Banach spaces, and let $H$ be a closed subspace of $\lxy$ containing the finite rank operators. If $\Xs$ and $Y$ are octahedral, then $H$ is octahedral.
\end{theorem}

For the convenience of reference, let us point out a property for a Banach space $X$ used as a hypothesis in Theorem~\ref{thm: L(X,Y) OH when Y has P1}:
\begin{itemize}
\item[(P)] there is a $u\in S_X$ such that the set
\[
\bigl\{\xs\in B_{\Xs}\colon\,\xs(u)=1\bigr\}
\]
is norming for $X$ in the sense that, for every $x\in S_X$ and every $\eps>0$, there is an $\xs\in B_{\Xs}$ such that
\[
|\xs(x)|>1-\eps\quad\text{and}\quad\xs(u)=1.
\]
\end{itemize}

\begin{theorem}[{\cite[Theorems~3.1 and 3.2]{becerra_guerrero_octahedral_2015}}]\label{thm: L(X,Y) OH when Y has P1}
Let $X$ and $Y$ be Banach spaces, and let $H$ be a closed subspace of $\lxy$ containing the finite rank operators.
\begin{itemize}
\item[{\rm(a)}]
If $\Xast$ is octahedral and $Y$ has property \emph{(P)}, then $H$ is octahedral.
\item[{\rm(b)}]
If $\Xast$ has property \emph{(P)} and $Y$ is octahedral, then $H$ is octahedral.
\end{itemize}

\end{theorem}

In Section~\ref{sec: necessary conditions}, we first establish a quantitative version of the following Theorem~\ref{thm: L(X,Y) OH implies Y is OH} from \cite{becerra_guerrero_octahedral_2015} in terms of average roughness (see Theorem~\ref{thm: L(X,Y) and Xs non-rough implies Y is rough}).

\begin{theorem}[{\cite[Proposition~3.9 and Corollary~3.10]{becerra_guerrero_octahedral_2015}}]\label{thm: L(X,Y) OH implies Y is OH}
Let $X$ and $Y$ be Banach spaces, and let $H$ be a closed subspace of $\lxy$ containing the finite rank operators.
\begin{itemize}
\item[{\rm(a)}]
If $H$ is octahedral and $\Xast$ is non-rough, then $Y$ is octahedral.
\item[{\rm(b)}]
If $H$ is octahedral and $Y$ is non-rough, then $\Xast$ is octahedral.
\end{itemize}

\end{theorem}
\noindent We then introduce a new notion of \emph{weak $\delta$-average roughness} of a Banach space which corresponds to the property that the diameter of every non-empty relatively weak$^\ast$ open subset of the dual unit ball is greater than or equal to $\delta$ (see Theorem~\ref{thm: omnibus thm for the weak* d-delta-P}). Our main result in Section~\ref{sec: necessary conditions} is a quantitative  version of Theorem~\ref{thm: L(X,Y) OH implies Y is OH} in terms of weak $\delta$-average roughness (see Theorem~\ref{thm: L(X,Y) and Xs non-rough implies Y is weakly rough}).  

Let us fix some more notation. Let $X$ and $Y$ be Banach spaces. For $x^\ast\in X^\ast$ and $y\in Y$, we denote by $x^\ast\otimes y$ the operator in $\mathcal L(X,Y)$ defined by $(x^\ast\otimes y)(x)=x^\ast(x)y$, $x\in X$. 
For a subset $A$ of $X$, its linear span and convex hull are denoted by $\operatorname{span}(A)$ and $\operatorname{conv}(A)$, respectively.

\section{Sufficient conditions for roughness in spaces of operators}\label{sec: sufficient conditions}

The main objective in this section is to relax the assumptions in Theorems~\ref{thm: L(X,Y) OH when Y OH} and \ref{thm: L(X,Y) OH when Y has P1}. In order to do so, we introduce a new notion of \emph{alternative octahedrality}, which in general is a weaker property than both octahedrality and property (P). 

\begin{definition}
Let $X$ be a Banach space. We say that $X$ is \emph{alternatively octahedral} if, whenever $n\in\N$, $x_1,\dots,x_n\in S_X$, and $\eps>0$, there is a $y\in S_X$ such that
\[
\max\{\|x_i+y\|,\|x_i-y\| \}\geq 2-\eps\quad\text{ for all $i\in\{1,\dots,n\}$}.
\]
\end{definition}

Note that the alternative octahedrality of $X$ is equivalent to the following condition:
\begin{itemize}
\item
whenever $n\in \N$, $x_1,\dotsc,x_n\in S_{X}$, and $\eps>0$, there are $y\in S_X$ and $\xs_1,\dotsc,\xs_n\in S_{\Xs}$ such that, for every $i\in\{1,\dotsc,n\}$,
\[
|\xs_i(x_i)|>1-\eps
\quad\text{and}\quad
|\xs_i(y)|>1-\eps.
\]
\end{itemize}
Observe that both octahedrality and the property (P) above imply alternative octahedrality. On the other hand, for example, $c_0$ is alternatively octahedral, but fails to be octahedral nor does it have property (P).

\begin{theorem}\label{thm: L(X,Y) OH when Y has P2}
Let $X$ and $Y$ be Banach spaces, let $H$ be a closed subspace of $\lxy$ containing the finite rank operators, and let $\delta>0$.
\begin{itemize}
\item[{\rm(a)}]
If  $\Xs$ is $\delta$-average rough and $Y$ is alternatively octahedral, then  $H$ is $\delta$-average rough.
\item[{\rm(b)}]
If $\Xs$ is alternatively octahedral and $Y$ is $\delta$-average rough, then $H$ is $\delta$-average rough.

\end{itemize}
\end{theorem}
\noindent Theorem~\ref{thm: L(X,Y) OH when Y has P2} improves Theorems~\ref{thm: L(X,Y) OH when Y OH} and \ref{thm: L(X,Y) OH when Y has P1}, because octahedrality is $2$-average roughness. In particular, it shows that $\calL(c_0,c_0)$ is octahedral, while its octahedrality can not be deduced from Theorems~\ref{thm: L(X,Y) OH when Y OH} or \ref{thm: L(X,Y) OH when Y has P1}. Also, Theorem~\ref{thm: L(X,Y) OH when Y has P2} allows one to refine \cite[Corollaries~3.3, 3.4, and 3.6]{becerra_guerrero_octahedral_2015}.

\begin{proof}[Proof of Theorem~\ref{thm: L(X,Y) OH when Y has P2}]
(a).
Let $n\in\mathbb N$, $S_1,\dotsc,S_n\in S_H$, and $\varepsilon>0$.
It suffices to find a $T\in H$ with $\|T\|=\varepsilon$ satisfying
\[
\frac{1}{n}\sum_{i=1}^n \bigl(\|S_i+T\|+\|S_i-T\|\bigr)>(\delta-5\varepsilon)\|T\|+2.
\]
Choose $x_i\in S_X$ so that $\|S_ix_i\|>1-\varepsilon^2$.
Since $Y$ is alternatively octahedral, there are $y\in S_Y$ and $y^\ast_1,\dotsc,y^\ast_n\in S_{Y^\ast}$ such that, for every $i\in\{1,\dotsc,n\}$,
\[
|y^\ast_i(S_ix_i)|>1-\varepsilon^2
\quad\text{and}\quad
y^\ast_i(y)>1-\varepsilon.
\]
Since $X^\ast$ is $\delta$-average rough, there is an $x^\ast\in X^\ast$ with $\|x^\ast\|=\varepsilon$ such that
\[
\frac{1}{n}\sum_{i=1}^n \bigl(\|S_i^\ast y_i^\ast+x^\ast\|+\|S_i^\ast y_i^\ast-x^\ast\|\bigr)>(\delta-\varepsilon)\|x^\ast\|+\frac{2}{n}\sum_{i=1}^n\|S_i^\ast y_i^\ast\|.
\]
Thus
\[
\frac{1}{n}\sum_{i=1}^n \bigl(\|S_i^\ast y_i^\ast+x^\ast\|+\|S_i^\ast y_i^\ast-x^\ast\|\bigr)>(\delta-3\varepsilon)\|x^\ast\|+2.
\]
Letting $T:=x^\ast\otimes y$, one has $\|T\|=\|x^\ast\|=\varepsilon$ and
\begin{multline*}
\frac{1}{n}\sum_{i=1}^n \bigl(\|S_i+T\|+\|S_i-T\|\bigr)\\
\begin{aligned}
&\geq \frac{1}{n}\sum_{i=1}^n \bigl(\|S_i^\ast y_i^\ast+T^\ast y_i^\ast\|+\|S_i^\ast y_i^\ast-T^\ast y_i^\ast\|\bigr)\\
&\geq \frac{1}{n}\sum_{i=1}^n \bigl(\|S_i^\ast y_i^\ast+x^\ast\|+\|S_i^\ast y_i^\ast-x^\ast\|-2\|x^\ast-T^\ast y_i^\ast\|\bigr)\\
&= \frac{1}{n}\sum_{i=1}^n \bigl(\|S_i^\ast y_i^\ast+x^\ast\|+\|S_i^\ast y_i^\ast-x^\ast\|-2\bigl(1-y_i^\ast(y)\bigr)\|x^\ast\|\bigr)\\
&>(\delta-3\varepsilon)\|x^\ast\|+2-2\varepsilon\|x^\ast\|\\
&=(\delta-5\varepsilon)\|T\|+2.
\end{aligned}
\end{multline*}

(b). The proof is similar to that of (a).
\end{proof}

\begin{theorem}\label{thm: L(X,Y) LOH when Y has P2}
Let $X$ and $Y$ be Banach spaces, let $H$ be a closed subspace of $\lxy$ containing the finite rank operators, and let $\delta>0$.
\begin{itemize}
\item[{\rm(a)}]
If $\Xs$ is $\delta$-rough, then $H$ is $\delta$-rough.
\item[{\rm(b)}]
If $Y$ is $\delta$-rough, then $H$ is $\delta$-rough.
\end{itemize}
\end{theorem}
\begin{proof}
\medskip
(a). We mimic the proof of Theorem~\ref{thm: L(X,Y) OH when Y has P2}(a) with $n=1$.
Let $S\in S_H$ and $\varepsilon>0$.
It suffices to find a $T\in H$ with $\|T\|=\varepsilon$ satisfying
\[
\|S+T\|+\|S-T\|>(\delta-5\varepsilon)\|T\|+2.
\]
Let $y^\ast\in S_{Y^\ast}$ be such that $\|\Sast\yast\|>1-\eps^2$. Let $y\in S_Y$ be such that $\yast(y)>1-\eps$.
Since $X^\ast$ is $\delta$-rough, there is an $x^\ast\in X^\ast$ with $\|x^\ast\|=\varepsilon$ such that
\[
\|S^\ast y^\ast+x^\ast\|+\|S^\ast y^\ast-x^\ast\|>(\delta-\varepsilon)\|x^\ast\|+2\|S^\ast y^\ast\|.
\]
Thus
\[
\|S^\ast y^\ast+x^\ast\|+\|S^\ast y^\ast-x^\ast\|>(\delta-3\varepsilon)\|x^\ast\|+2.
\]
Letting $T:=x^\ast\otimes y$, one has $\|T\|=\|x^\ast\|=\varepsilon$ and
\begin{multline*}
\quad\|S+T\|+\|S-T\| \\
\begin{aligned}
&\geq \|S^\ast y^\ast+T^\ast y^\ast\|+\|S^\ast y^\ast-T^\ast y^\ast\|\\
&\geq \|S^\ast y^\ast+x^\ast\|+\|S^\ast y^\ast-x^\ast\|-2\|x^\ast-T^\ast y^\ast\|\quad\\
&= \|S^\ast y^\ast+x^\ast\|+\|S^\ast y^\ast-x^\ast\|-2\bigl(1-y^\ast(y)\bigr)\|x^\ast\|\quad\\
&>(\delta-3\varepsilon)\|x^\ast\|+2-2\varepsilon\|x^\ast\|\\
&=(\delta-5\varepsilon)\|T\|+2.
\end{aligned}
\end{multline*}

(b). The proof is similar to that of (a).
\end{proof}

We do not know whether for the octahedrality of $\calL(X,Y)$ it is, in general, sufficient that only one of the spaces $\Xast$ or $Y$ is octahedral without any additional assumptions (see also the discussion after Corollary~3.6 in \cite{becerra_guerrero_octahedral_2015}). We next show that, for $1<p<\infty$, the space $\calL(c_0, \ell_p^2)$ is octahedral. Its octahedrality can not be deduced from Theorem~\ref{thm: L(X,Y) OH when Y has P2}, because $\ell^2_p$ is not alternatively octahedral. The octahedrality of $\calL(c_0, \ell_1^n)$ and $\calL(c_0, \ell_\infty^n)$ is a direct consequence of Theorem~\ref{thm: L(X,Y) OH when Y has P2}, because $\ell^n_1$ and $\ell^n_\infty$ are both alternatively octahedral.

\begin{proposition}\label{prop: L(c0,l22) is OH}
If $1< p<\infty$, then $\calL(c_0, \ell_p^2)$ is octahedral.
\end{proposition}

Our proof of Proposition~\ref{prop: L(c0,l22) is OH} uses the following lemma.


\begin{lemma}\label{lem: pm circle}
Let $1<p<\infty$, let $n\in\N$, and let $a_1=(\alpha_1,\beta_1),\dotsc,a_n=(\alpha_n,\beta_n)\in S_{\ell^2_p}$ be such that $\alpha_1,\dotsc,\alpha_n\geq 0$ and $\beta_1\geq\dotsb\geq\beta_n$. 
Then
\[
\theta_1\cdot\frac{a_1+a_n}{2}+\theta_2\cdot\frac{a_2-a_1}{2}+\dotsb+\theta_n\cdot\frac{a_n-a_{n-1}}{2}\in B_{\ell^2_p}
\]
for all $\theta_1,\dotsc,\theta_n\in\{-1,1\}$.
\end{lemma}
\begin{proof}
Let $\theta_1,\dotsc,\theta_n\in\{-1,1\}$. Put
\[
x:=\theta_1\cdot\frac{a_1+a_n}{2}+\theta_2\cdot\frac{a_2-a_1}{2}+\dotsb+\theta_n\cdot\frac{a_n-a_{n-1}}{2}.
\]
We will show that $x\in B_{\ell^2_p}$. Without loss of generality we may assume that $\theta_1=1$. Since
\[
\frac{a_n}{2}=\frac{a_1}{2}+\frac{a_2-a_1}{2}+\dotsb+\frac{a_n-a_{n-1}}{2},
\]
we have that
\begin{align*}
x=a_1&+\frac{a_2-a_1}{2}+\dotsb+\frac{a_n-a_{n-1}}{2}+\\
&+\theta_2\cdot\frac{a_2-a_1}{2}+\dotsb+\theta_n\cdot\frac{a_n-a_{n-1}}{2}.
\end{align*}
Hence there is an odd number of increasing indices $k_1,\dotsc, k_{2l+1}$ such that $x$ is representable as
\begin{equation}\label{eq: x esitus}
x=a_{k_{1}}-a_{k_{2}}+a_{k_{3}}-\dotsb -a_{k_{2l}}+a_{k_{2l+1}}.
\end{equation}

To show that $x\in B_{\ell^2_p}$, we use the following geometric properties of $\ell^2_p$.
\begin{fact}\label{fact: Theta}
For $a,b\in S_{\ell^2_p}$, let $\Theta_{a,b}:=B_{\ell_p^2}\cap (B_{\ell_p^2}+(a+b))$.
	\begin{itemize}
		\item[{\rm(a)}]
	    If $a,b\in S_{\ell^2_p}$ and $y\in \Theta_{a,b}$, then $a-y+b\in\Theta_{a,b}$.
		\item[{\rm(b)}] If $a$, $b$, and $c$ are pairwise different elements of $S_{\ell^2_p}$ and $b\in\Theta_{a,c}$, then $\Theta_{a,b}\subset\Theta_{a,c}$.
		\end{itemize}
\end{fact}

Since $a_{k_{l+1}}\in\Theta_{a_{k_{l}}, a_{k_{l+2}}}$, we have that $z:=a_{k_{l}}-a_{k_{l+1}}+a_{k_{l+2}}\in \Theta_{a_{k_{l}}, a_{k_{l+2}}}$ by part (a) of Fact. We can write the middle part of the right hand side of (\ref{eq: x esitus}) as
\[
\dots a_{k_{l-1}}-(a_{k_{l}}-a_{k_{l+1}}+a_{k_{l+2}})+a_{k_{l+3}}\dots
\]
By part (b) of Fact, $z\in \Theta_{a_{k_{l}}, a_{k_{l+2}}}\subset \Theta_{a_{k_{l-1}}, a_{k_{l+3}}}$. Applying part (a) of Fact we have that $a_{k_{l-1}}-z+a_{k_{l+3}}\in \Theta_{a_{k_{l-1}}, a_{k_{l+3}}}$. Continuing in this way, we will finally have $x\in \Theta_{a_{k_{1}}, a_{k_{2l+1}}}\subset B_{\ell^2_p}$.
\end{proof}

\begin{proof}[Proof of Proposition~ \ref{prop: L(c0,l22) is OH}]
	Let $n\in\N$, $S_1,\dots,S_n\in S_{\calL(c_0, \ell_p^2)}$, and $\varepsilon\in(0,1)$. It suffices to show that there is a $T\in S_{\calL(c_0, \ell_p^2)}$ such that 
	\[
	\|S_i+T\|\geq 2-3\eps\quad \text{ for all $i\in\{1,\dots,n\}$}.
	\]
	
	Choose $x_i\in S_{c_0}$ such that $\|S_ix_i\|\geq1-\eps$. Without loss of generality we may assume that $x_1,\dots,x_n$ are finitely supported, that is, there is a $N_1\in\N$ such that $x_1,\dots,x_n\in\spn\{e_1,\dots,e_{N_1}\}$.
	
	Since $S_1,\dots, S_n$ are finite rank operators and $(e_k)$ is a weakly null sequence in $c_0$, there is a $N_2\in\N$ such that $\|S_ie_k\|\leq \varepsilon/n$ for all $i\in\{1,\dots,n\}$ and $k\geq N_2$. Take $N=\max\{N_1,N_2\}$.
	
	For all $i\in\{1,\dots,n\}$, put $a_i:=S_ix_i/\|S_ix_i\|$. By reordering $a_1,\dots,a_n$ and by replacing $a_i$ with $-a_i$ if necessary, we may assume that $a_1,\dots,a_n$ satisfy the assumptions of Lemma~\ref{lem: pm circle}. 
	
	Define $T\colon c_0\rightarrow \ell_p^2$ by
	\[
	Te_{N+1}=\frac{a_1+a_n}{2},\quad Te_{N+2}=\frac{a_2-a_1}{2},\quad \dotsc,\quad Te_{N+n}=\frac{a_n-a_{n-1}}{2},
	\]
	and $Te_k=0$, if $k\in\N\setminus\{N+1, \dots, N+n\}$.
	
	By Lemma~\ref{lem: pm circle}, $\|T\|\leq 1$. On the other hand, $\|T\|\geq 1$, because $T(e_{N+1}+\dots+e_{N+n})=a_n$. Thus $\|T\|=1$.
	
	Fix $i\in\{1,\dots,n\}$. Choose $\theta_1,\dots,\theta_n\in\{-1,1\}$ so that
	\[
	\theta_1\cdot\frac{a_1+a_n}{2}+\theta_2\cdot\frac{a_2-a_1}{2}+\dots+\theta_n\cdot\frac{a_n-a_{n-1}}{2}=a_i.
	\]
	Let $y_i:=\theta_1e_{N+1}+\dots+\theta_n e_{N+n}$. Since  $Ty_i=a_i=S_ix_i/\|S_ix_i\|$, $Tx_i=0$, and $x_i+y_i\in S_{c_0}$, we get
	\begin{align*}
	\|S_i+T\|&\geq \|(S_i+T)(x_i+y_i)\|\\
	&=\|S_ix_i+S_iy_i+Ty_i\|\\
	&\geq \|S_ix_i+Ty_i\|-\|S_iy_i\|\\
	&\geq 2\|S_ix_i\|-\eps\geq 2-3\eps.
	\end{align*}
\end{proof}

\section{Necessary conditions for roughness in spaces of operators}\label{sec: necessary conditions}

In this section, we first prove a quantitative version of Theorem~\ref{thm: L(X,Y) OH implies Y is OH} in terms of roughness. Our main result is a quantitative version of Theorem~\ref{thm: L(X,Y) OH implies Y is OH} for weakly octahedral Banach spaces.

Recall that a Banach space is non-rough if and only if its dual unit ball has weak$^\ast$ slices of arbitrarily small diameter \cite[Proposition~1]{john_rough_1978}.

\begin{theorem}\label{thm: L(X,Y) and Xs non-rough implies Y is rough}
Let $X$ and $Y$ be Banach spaces, let $H$ be a closed subspace of $\lxy$ containing the finite rank operators, and let $\delta>0$.
\begin{itemize}
\item[{\rm(1)}]
Let $H$ be $\delta$-average rough.
\begin{itemize}
\item[{\rm(a)}]
If $\Xs$ is non-rough, then $Y$ is $\delta$-average rough.
\item[{\rm(b)}]
If $Y$ is non-rough, then $\Xs$ is $\delta$-average rough.
\end{itemize}
\item[{\rm(2)}]
Let $H$ be $\delta$-rough.
\begin{itemize}
\item[{\rm(a)}]
If $\Xs$ is non-rough, then $Y$ is $\delta$-rough.
\item[{\rm(b)}]
If $Y$ is non-rough, then $\Xs$ is $\delta$-rough.
\end{itemize}
\end{itemize}
\end{theorem}
\begin{proof}
(1a).
Let $n\in\N$, $y_1,\dots,y_n\in S_Y$, and $\eps\in(0,1/3)$. For the $\delta$-average roughness of $Y$, it suffices to find a $z\in Y$ with $\|z\|<\eps$ such that
\[
\frac{1}{n}\sum_{i=1}^{n}\bigl(\|y_i+z\|+\|y_i-z\|\bigr)>2+(\delta-5\eps)\|z\|.
\]

Since $\Xs$ is non-rough, there are $\xs\in S_{\Xs}$ and $\alpha\in(0,3\eps)$ such that,
for the slice $S(\xs,\alpha):=\{x\in B_X\colon\,\xs(x)>1-\alpha\}$, one has $\diam S(\xs,\alpha)<\eps$.

Let $S_i:=\xast\otimes y_i$ for every $i\in\{1,\dots,n\}$. Since $H$ is $\delta$-average rough, there is a $T\in H$ with $\|T\|<\dfrac\alpha3$ such that
\[
\frac{1}{n}\sum_{i=1}^{n}\bigl(\|S_i+T\|+\|S_i-T\|\bigr)>2+(\delta-\eps)\|T\|.
\]
For $j\in\{1,2\}$, choosing $x_{i,j}\in S_X$ and $\ys_{i,j}\in S_{\Ys}$ so that
\begin{align*}
	\xs(x_{i,j})\,\ys_{i,j}(y_i)+(-1)^{j+1}\ys_{i,j}(Tx_{i,j})&=\ys_{i,j}(S_ix_{i,j}+(-1)^{j+1}Tx_{i,j})\\
	&\geq \|S_i+(-1)^{j+1}T\|-\eps\|T\|,
\end{align*}
one may assume that both $\xs(x_{i,j})>0$ and $\ys_{i,j}(y_i)>\nobreak0$, and thus $x_{i,j}\in S(\xs,\alpha)$, because
	\begin{align*}
	\xast(x_{i,j})&\geq\|S_i+(-1)^{j+1}T\|-\eps\|T\|-|\ys_{i,j}(Tx_{i,j})|\\
&\geq\|S_i\|-\|T\|-\eps\|T\|-\|T\|\geq\|S_i\|-3\|T\|>1-\alpha.\\
		\end{align*}
		
	Letting $z:=Tx_{1,1}$, we have that $\|z\|\leq\|T\|<\dfrac\alpha3<\eps$ and, for every $i\in\{1,\dots,n\}$,
	\begin{align*}
	\|y_i+z\|&+\|y_i-z\|=\|y_i+Tx_{1,1}\|+\|y_i-Tx_{1,1}\|\\
	&\geq \|y_i+Tx_{i,1}\|+\|y_i-Tx_{i,2}\|\\
	&\quad-\|Tx_{1,1}-Tx_{i,1}\|-\|Tx_{1,1}-Tx_{i,2}\|\\
	&\geq \yast_{i,1}(y_i+Tx_{i,1})+\yast_{i,2}(y_i-Tx_{i,2})\\
	&\quad-\|T\|\,\|x_{1,1}-x_{i,1}\|-\|T\|\,\|x_{1,1}-x_{i,2}\|\\
	&\geq \xast(x_{i,1})\yast_{i,1}(y_i)+\yast_{i,1}(Tx_{i,1})\\
	&\quad+\xast(x_{i,2})\yast_{i,2}(y_i)-\yast_{i,2}(Tx_{i,2})-2\eps\|T\|\\
	&\geq \|S_i+T\|-\eps\|T\|+\|S_i-T\|-\eps\|T\|-2\eps\|T\|\\
	&= \|S_i+T\|+\|S_i-T\|-4\eps\|T\|,\\
	\end{align*}
	
	and thus
	
	\begin{align*}
	\frac{1}{n}\sum_{i=1}^{n}\big(\|y_i+z\|+\|y_i-z\|\big)&\geq \frac{1}{n}\sum_{i=1}^{n}\big(\|S_i+T\|+\|S_i-T\|\big)-4\eps\|T\|\\
	&>2+(\delta-5\eps)\|T\|\geq 2+(\delta-5\eps)\|z\|.
	\end{align*}
	
(1b). The proof is similar to that of (1a).	
\medskip

(2). The proof is exactly that of (1) with $n=1$.
\end{proof}

Note that Theorem~\ref{thm: L(X,Y) and Xs non-rough implies Y is rough}(1) with $\delta=2$ is exactly Theorem~\ref{thm: L(X,Y) OH implies Y is OH}. Combining Theorems~\ref{thm: L(X,Y) OH when Y has P2} and \ref{thm: L(X,Y) and Xs non-rough implies Y is rough} improves \cite[Corollary~3.11]{becerra_guerrero_octahedral_2015} as follows.

\begin{corollary}
Let $X$ and $Y$ be Banach spaces, and let $H$ be a closed subspace of $\lxy$ containing the finite rank operators. Suppose that $\Xs$ is non-rough and alternatively octahedral. The following assertions are equivalent:
\begin{itemize}
\item[{\rm(i)}] $H$ is octahedral;
\item[{\rm(ii)}] $Y$ is octahedral.
\end{itemize}
\end{corollary}

The following definition is motivated by the known dual characterizations of (average) roughness in terms of the diameter of (convex combinations of) weak$^\ast$ slices (see Introduction). 

\begin{definition}
Let $X$ be a Banach space and $\delta>0$. We say that the space $X$ is \emph{weakly $\delta$-average rough} if every non-empty relatively weak$^{\ast}$ open subset of $B_{\Xs}$ has diameter greater than or equal to $\delta$.
\end{definition}

\begin{theorem}\label{thm: omnibus thm for the weak* d-delta-P}
Let $\delta>0$.
The following assertions are equivalent:
\begin{itemize}
\item[{\rm(i)}]
$X$ is weakly $\delta$-average rough;
%
\item[{\rm(ii)}]
whenever $n\in\N$, $x_1,\dotsc,x_n\in S_X$,  $\xs\in B_{\Xs}$, and $\eps, t_0>\nobreak0$,
there is a $y\in S_X$ such that,
\[
\inf_{\substack{i\in\{1,\dotsc,n\}\\t\geq t_0}}\frac{\|x_i+ty\|-\xs(x_i)}{t}-\sup_{\substack{i\in\{1,\dotsc,n\}\\t\geq t_0}}\frac{\|x_i-ty\|-\xs(x_i)}{-t}>\delta-\eps;
\]
\item[{\rm(iii)}]
whenever $E$ is a finite-dimensional subspace of $X$,  $\xs\in B_{\Xs}$, and $\eps,t_0>\nobreak0$,
there is a $y\in S_X$ such that
\begin{equation}\label{eq: weakly locally uniformly delta-rough ekv. tingimus 2}
\inf_{\substack{x\in S_E\\t\geq t_0}}\frac{\|x+ty\|-\xs(x)}{t}-\sup_{\substack{x\in S_E\\t\geq t_0}}\frac{\|x-ty\|-\xs(x)}{-t}>\delta-\eps;
\end{equation}
\item[{\rm(iv)}]
whenever $E$ is a finite-dimensional subspace of $X$, $\xs\in B_{\Xs}$, and $\eps>\nobreak0$,
there are $\xs_1,\xs_2\in\Xs$ with $\|\xs_1\|,\|\xs_2\|\leq 1+\eps$, and $y\in S_X$ satisfying
\begin{equation*}
\xs_1|_E=\xs_2|_E=\xs|_E\quad\text{and}\quad \xs_1(y)-\xs_2(y)>\delta-\eps;
\end{equation*}
\item[{\rm(v)}]
whenever $n\in\N$, $x_1,\dotsc,x_n\in S_X$, $\xs\in S_{\Xs}$, and $\eps>0$,
there are $\xs_1,\xs_2\in B_{\Xs}$ and $y\in B_X$ satisfying
\[
\bigl|\xs_j(x_i)-\xs(x_i)\bigr|<\eps\quad\text{for all $i\in\{1,\dotsc,n\}$ and $j\in\{1,2\}$,}
\]
and $\xs_1(y)-\xs_2(y)>\delta-\eps$.
\end{itemize}
\end{theorem}

The proof of the implication (iii)$\Rightarrow$(iv) of Theorem \ref{thm: omnibus thm for the weak* d-delta-P} relies on the following lemma which explains the role of the $\inf$ and $\sup$ in the condition (iii).

\begin{lemma}\label{lem: f-naali jatkamine E-lt span(E cup {y})-le}
Let $E$ be a finite-dimensional subspace of $X$, $y\in S_X\setminus E$, $\xs\in B_{\Xs}$, $t_0\in(0,1)$, and
$\gamma$ is such that
\[
a:=\sup_{\substack{x\in S_E\\t\geq t_0}}\frac{\|x-ty\|-\xs(x)}{-t}
\leq
\gamma
\leq
\inf_{\substack{x\in S_E\\t\geq t_0}}\frac{\|x+ty\|-\xs(x)}{t}=:b.
\]
Then, for the functional
\[
g\colon\spann\bigl(E\cup\{y\}\bigr)\ni x+ty\longmapsto \xs(x)+t\gamma,\quad x\in E,\, t\in\mathbb R,
\]
one has $\|g\|\leq\frac{1+t_0}{1-t_0}$.
\end{lemma}
\begin{proof}
First observe that $|\gamma|\leq1$, because, letting $t\to\infty$, one obtains $a\geq-1$ and $b\leq1$.

In order that $\|g\|\leq\frac{1+t_0}{1-t_0}$, it suffices to show that, letting $x\in S_E$ and $t>0$, one has
\[
|\xs(x)+t\,\gamma|\leq\frac{1+t_0}{1-t_0}\,\|x+t\,y\|.
\]
For $0<t\leq t_0$, one has
\begin{align*}
\bigl|\xs(x)+t\gamma\bigr|
&\leq1+t=\frac{1+t}{1-t}\bigl(\|x\|-t\,\|y\|\bigr)\leq\frac{1+t}{1-t}\,\|x+ty\|\\
&\leq\frac{1+t_0}{1-t_0}\,\|x+ty\|.
\end{align*}
For $t\geq t_0$, one has, in fact, a stronger inequality:
\[
|\xs(x)+t\,\gamma|\leq\|x+t\,y\|.
\]
Indeed, the latter condition is equivalent to
\[
\dfrac{\|-x-ty\|-\xs(-x)}{-t}\leq\gamma\leq\dfrac{\|x+ty\|-\xs(x)}{t}
\]
which holds because $\gamma\in[a,b]$.
\end{proof}

\begin{proof}[Proof of Theorem \ref{thm: omnibus thm for the weak* d-delta-P}]
(i)$\Leftrightarrow$(v) and (iv)$\Rightarrow$(v) are obvious.

\medskip
(v)$\Rightarrow$(ii).
Assume that (v) holds. Let $n\in\N$, $x_1,\dotsc,x_n\in S_X$, $\xs\in B_{\Xs}$, and $\eps, t_0>0$.
By (v), there are $\us,\vs\in B_{\Xs}$ and $y\in S_X$ such that, for every $i\in\{1,\dotsc,n\}$,
\[
|\us(x_i)-\xs(x_i)|<\eps t_0
\quad
\text{and}\quad
|\vs(x_i)-\xs(x_i)|<\eps t_0,
\]
and
\[
\vs(y)-\us(y)>\delta-\eps.
\]
Now let $i\in\{1,\dotsc,n\}$ and $t\geq t_0$ be arbitrary. Then
\begin{align*}
\vs(y)
&=\frac{\vs(x_i+ty)-\vs(x_i)}{t}\leq\frac{\|x_i+ty\|-\vs(x_i)}{t}\\
&<\frac{\|x_i+ty\|-\xs(x_i)+\eps t_0}{t}
\leq\frac{\|x_i+ty\|-\xs(x_i)}{t}+\eps
\end{align*}
and
\begin{align*}
-\us(y)&=\frac{\us(x_i-ty)-\us(x_i)}{t}\leq\frac{\|x_i-ty\|-\us(x_i)}{t}\\
&<\frac{\|x_i-ty\|-\xs(x_i)+\eps t_0}{t}
\leq\frac{\|x_i-ty\|-\xs(x_i)}{t}+\eps.
\end{align*}
It follows that
\begin{align*}
\inf_{\substack{i\in\{1,\dotsc,n\}\\t\geq t_0}}&\frac{\|x_i+ty\|-\xs(x_i)}{t}-\sup_{\substack{i\in\{1,\dotsc,n\}\\t\geq t_0}}\frac{\|x_i-ty\|-\xs(x_i)}{-t}\\
&\geq\vs(y)-\eps-\us(y)-\eps>\delta-3\eps.
\end{align*}

\medskip
(ii)$\Rightarrow$(iii).
Let $E$ be a finite-dimensional subspace of $X$,  $\xs\in B_{\Xs}$, and $\eps, t_0>\nobreak0$.
It suffices to observe that, letting $y\in S_X$ be produced by (ii) where $\{x_1,\dotsc,x_n\}\subset S_X$ is an $\eps t_0$-net for $S_E$,
the difference of the $\inf$ and $\sup$ in (iii) is greater than $\delta-5\eps$.

\medskip
(iii)$\Rightarrow$(iv).
Assume that (iii) holds. Let $E$ be a finite-dimensional subspace of $X$, $\xs\in B_{\Xs}$, and $\eps>0$.
Choosing $t_0\in(0,1)$ so that $\frac{1+t_0}{1-t_0}<1+\eps$,
let $y\in S_X$ satisfy (\ref{eq: weakly locally uniformly delta-rough ekv. tingimus 2}).
Define $g_1,g_2\in\bigl(\spann(E\cup\{y\})\bigr)^\ast$ in the same manner as $g$ in Lemma \ref{lem: f-naali jatkamine E-lt span(E cup {y})-le}
where, respectively, $\gamma=b$ and $\gamma=a$.
One may let the desired $\xs_1$ and $\xs_2$ be any norm preserving extensions to $X$ of $g_1$ and $g_2$, respectively.
\end{proof}

\begin{theorem}\label{thm: L(X,Y) and Xs non-rough implies Y is weakly rough}
Let $X$ and $Y$ be Banach spaces, let $H$ be a closed subspace of $\lxy$ containing the finite rank operators, and let $\delta>0$. Suppose that $H$ is weakly $\delta$-average rough.
\begin{itemize}

\item[{\rm(a)}]
If $\Xs$ is non-rough, then $Y$ is weakly $\delta$-average rough.
\item[{\rm(b)}]
If $Y$ is non-rough, then $\Xs$ is weakly $\delta$-average rough.
\end{itemize}
\end{theorem}

\begin{proof}
(a).
Let $n\in\N$, $y_1,\dotsc,y_n\in S_Y$, $\ys\in S_{\Ys}$, and $\varepsilon\in(0,1)$. For the weak $\delta$-average roughness of $Y$, by Theorem~\ref{thm: omnibus thm for the weak* d-delta-P}, it suffices to find $y^\ast_1, y^\ast_2\in B_{Y^\ast}$ and $y\in B_Y$ such that 
\[
\bigl|\ys_j(y_i)-\ys(y_i)\bigr|<6\eps\quad\text{for all $i\in\{1,\dotsc,n\}$ and $j\in\{1,2\}$,}
\]
and
\[
\ys_1(y)-\ys_2(y)>\delta-13\eps.
\]

Since $\Xs$ is non-rough, there are $\xs\in S_{\Xs}$ and $\alpha\in(0,\eps)$ such that,
for the slice $S(\xs,\alpha):=\{x\in B_X\colon\,\xs(x)>1-\alpha\}$, one has $\diam S(\xs,\alpha)<\eps$.

Choose $\xss\in S_{\Xss}$ and $y_0\in S_Y$ so that $\xss(\xs)=1$ and $\ys(y_0)>1-\alpha^2$, and put
$S_i:=\xs\otimes y_i\in H$ for every $i\in\{0,1,\dotsc,n\}$, and $\phi:=\xss\otimes\ys\in H^\ast$, where
\[
(\xss\otimes\ys)(S)=\xss(S^\ast \ys),\quad S\in H.
\]
Since $H$ is weakly $\delta$-average rough, by Theorem~\ref{thm: omnibus thm for the weak* d-delta-P}, there are $\phi_1,\phi_2\in H^\ast$ with $\|\phi_1\|,\|\phi_2\|< 1+\alpha^2$, and $T\in S_H$ such that
\[
\phi_1(S_i)=\phi_2(S_i)=\phi(S_i)=\ys(y_i)\quad\text{for all $i\in\{0,1,\dotsc,n\}$,}
\]
and
\[
\phi_1(T)-\phi_2(T)>\delta-\eps.
\]
Denote by $B_X\otimes B_{\Ys}$ the set of functionals in $H^\ast$ of the form $x\otimes\ys$, where $x\in B_X$ and $\ys\in\Ys$.
Since, for every $S\in H$, there is some $f$ in the weak$^\ast$ closure of $B_X\otimes B_{\Ys}$ in $H^\ast$ 
such that $f(S)=\|S\|$, by the Hahn--Banach separation theorem,
it quickly follows that $\conv(B_X\otimes B_{\Ys})$ is weak$^\ast$ dense in $B_{H^\ast}$. 
Thus, for all $j\in\{1,2\}$, observing that $\bigl\|\frac1{\|\phi_j\|}\phi_j-\phi_j\bigr\|<\alpha^2$, there are
\[
\psi_j:=\sum_{k=1}^{m_j}\lambda_{j,k}x_{j,k}\otimes\ys_{j,k}\in\conv(B_X\otimes B_{\Ys})
\]
such that $\bigl|\psi_j(T)-\phi_j(T)\bigr|<\alpha^2$ and, for all $i\in\{0,1,\dotsc,n\}$,
\[
\biggl|\sum_{k=1}^{m_j}\lambda_{j,k}\xs(x_{j,k})\ys_{j,k}(y_i)-\ys(y_i)\biggr|=\bigl|\psi_j(S_i)-\phi_j(S_i)\bigr|<\alpha^2.
\]
Letting
\[
K_j:=\bigl\{k\in\{1,\dotsc,m_j\}\colon\, \xs(x_{j,k})>1-\alpha\bigr\}
\quad\text{and}\quad
\lambda_j:=\sum_{k\notin K_j}\lambda_{j,k},
\]
one has
\begin{align*}
1-2\alpha^2
&<\ys(y_0)-\alpha^2
<\sum_{k=1}^{m_j}\lambda_{j,k}\xs(x_{j,k})\ys_{j,k}(y_0)\\
&=\sum_{k\notin K_j}\lambda_{j,k}\xs(x_{j,k})\ys_{j,k}(y_0)+\sum_{k\in K_j}\lambda_{j,k}\xs(x_{j,k})\ys_{j,k}(y_0)\\
&\leq(1-\alpha)\sum_{k\notin K_j}\lambda_{j,k}+\sum_{k\in K_j}\lambda_{j,k}
=(1-\alpha)\,\lambda_j+1-\lambda_j\\
&=1-\alpha\,\lambda_j
\end{align*}
whence $\lambda_j<2\alpha$.
Now, putting $\ys_j:=\sum_{k=1}^{m_j} \lambda_{j,k}\ys_{j,k}$, one has
\begin{align*}
\bigl|\ys_j(y_i)-\ys(y_i)\bigr|
&=\biggl|\sum_{k=1}^{m_j} \lambda_{j,k}\ys_{j,k}(y_i)-\ys(y_i)\biggr|\\
&\leq\biggl|\sum_{k=1}^{m_j} \lambda_{j,k}\xs(x_{j,k})\ys_{j,k}(y_i)-\ys(y_i)\biggr|\\
&\quad\quad +\sum_{k\notin K_j} \lambda_{j,k}\bigl|1-\xs(x_{j,k})\bigr|\,\bigl|\ys_{j,k}(y_i)\bigr|\\
&\quad\quad +\sum_{k\in K_j} \lambda_{j,k}\bigl|1-\xs(x_{j,k})\bigr|\,\bigl|\ys_{j,k}(y_i)\bigr|\\
&<\alpha^2+2\lambda_j+\alpha<6\alpha<6\eps.
\end{align*}
Letting $x\in S(\xs,\alpha)$ be arbitrary, one has
\[
\|x-x_{j,k}\|<\eps\quad\text{for all $k\in K_j$,}
\]
thus
\begin{align*}
\bigl|\ys_j(Tx)-\psi_j(T)\bigr|
&=\biggl|\sum_{k=1}^{m_j}\lambda_{j,k}(\Ts\ys_{j,k})(x-x_{j,k})\biggr|\\
&\leq\sum_{k\notin K_j}\lambda_{j,k}\|x-x_{j,k}\|+\sum_{k\in K_j}\lambda_{j,k}\|x-x_{j,k}\|\\
&\leq2\lambda_j+\eps<4\alpha+\eps<5\eps,
\end{align*}
and it follows that
\begin{align*}
\ys_1(Tx)-\ys_2(Tx)
&=\ys_1(Tx)-\psi_1(T)+\psi_1(T)-\phi_1(T)+\phi_1(T)-\phi_2(T)\\
&\quad\quad+\phi_2(T)-\psi_2(T)+\psi_2(T)-\ys_2(Tx)\\
&\geq\phi_1(T)-\phi_2(T)\\
&\quad\quad-\bigl|\ys_1(Tx)-\psi_1(T)\bigr|-\bigl|\psi_1(T)-\phi_1(T)\bigr|\\
&\quad\quad-\bigl|\phi_2(T)-\psi_2(T)\bigr|-\bigl|\psi_2(T)-\ys_2(Tx)\bigr|\\
&>\delta-\eps-5\eps-\alpha^2-\alpha^2-5\eps>\delta-13\eps.
\end{align*}
\smallskip
(b). The proof is similar to that of (a).
\end{proof}

Taking $\delta=2$ in Theorem~\ref{thm: L(X,Y) and Xs non-rough implies Y is weakly rough} we get the following corollary.

\begin{corollary}\label{cor: H is WOH implies Y is WOH}
Let $X$ and $Y$ be Banach spaces, and let $H$ be a closed subspace of $\lxy$ containing the finite rank operators. Suppose that $H$ is weakly octahedral.
\begin{itemize}
\item[{\rm(a)}]
If $\Xs$ is non-rough, then $Y$ is weakly octahedral.
\item[{\rm(b)}]
If $Y$ is non-rough, then $\Xs$ is weakly octahedral.
\end{itemize}
\end{corollary}
\bibliographystyle{amsplain}

\end{document}